\documentclass[11pt,onecolumn]{article}

\usepackage[a4paper,margin=1.15in]{geometry}
\usepackage[T1]{fontenc}
\usepackage{lmodern}
\usepackage{microtype}
\usepackage[dvipsnames]{xcolor}
\usepackage{amsmath}
\usepackage{amssymb}
\usepackage{amsthm}
\usepackage{aliascnt}
\usepackage{mathtools}
\usepackage{enumitem}
\usepackage{titlesec}
\usepackage{fancyhdr}
\usepackage{hyperref}
\usepackage{cleveref}

\definecolor{linkblue}{HTML}{1F4E79}
\hypersetup{
    colorlinks=true,
    linkcolor=linkblue,
    citecolor=linkblue,
    urlcolor=linkblue,
    pdfborder={0 0 0}
}

\titleformat{\section}
    {\large\bfseries\color{linkblue}}{\thesection}{0.75em}{}
\titleformat{\subsection}
    {\normalsize\bfseries\color{linkblue}}{\thesubsection}{0.75em}{}
\titlespacing*{\section}{0pt}{2.3ex plus 0.8ex minus 0.2ex}{1.1ex plus 0.2ex}
\titlespacing*{\subsection}{0pt}{1.8ex plus 0.6ex minus 0.2ex}{0.8ex plus 0.2ex}

\setlist[itemize]{leftmargin=2.2em,itemsep=0.25ex,topsep=0.6ex}
\setlist[enumerate]{leftmargin=2.2em,itemsep=0.25ex,topsep=0.6ex}
\renewcommand{\headrulewidth}{0.4pt}
\renewcommand{\headrule}{\hbox to\headwidth{\color{linkblue}\leaders\hrule height \headrulewidth\hfill}}
\fancypagestyle{plain}{%
    \fancyhf{}
    \fancyfoot[C]{\small\color{linkblue}\thepage}
    \renewcommand{\headrulewidth}{0pt}
}

\newtheoremstyle{blueplain}
    {0.9ex}{0.9ex}{\itshape}{}%
    {\bfseries\color{linkblue}}{.}{0.55em}{}
\theoremstyle{blueplain}
\newtheorem{theorem}{Theorem}[section]

\newaliascnt{corollary}{theorem}
\newtheorem{corollary}[corollary]{Corollary}
\aliascntresetthe{corollary}

\newaliascnt{proposition}{theorem}
\newtheorem{proposition}[proposition]{Proposition}
\aliascntresetthe{proposition}

\newaliascnt{lemma}{theorem}

\aliascntresetthe{lemma}

\newtheoremstyle{bluedefinition}
    {0.9ex}{0.9ex}{\normalfont}{}%
    {\bfseries\color{linkblue}}{.}{0.55em}{}
\theoremstyle{bluedefinition}
\newaliascnt{remark}{theorem}

\aliascntresetthe{remark}

\crefname{theorem}{theorem}{theorems}
\Crefname{theorem}{Theorem}{Theorems}
\crefname{corollary}{corollary}{corollaries}
\Crefname{corollary}{Corollary}{Corollaries}
\crefname{proposition}{proposition}{propositions}
\Crefname{proposition}{Proposition}{Propositions}
\crefname{lemma}{lemma}{lemmas}
\Crefname{lemma}{Lemma}{Lemmas}
\crefname{remark}{remark}{remarks}
\Crefname{remark}{Remark}{Remarks}

\title{Subhomogeneity and Arveson Boundary of Free Polyhedra}
\author{\textsc{Tim Netzer}\\[0.25em]
{\small Department of Mathematics, University of Innsbruck}}

\date{\today}

\makeatletter
\renewcommand{\maketitle}{%
    \thispagestyle{plain}%
    \begin{center}
        \vspace*{-2.6em}
        {\LARGE\bfseries\color{linkblue}\@title\par}
        \vspace{1.25em}
        {\large\@author\par}
        \vspace{0.7em}
        {\small\color{gray}\@date\par}
    \end{center}
    \vspace{1.1em}
}
\makeatother

\renewenvironment{abstract}{%
    \begin{center}
    \begin{minipage}{0.9\textwidth}
    \small
    \begin{center}
        {\bfseries\color{linkblue}\abstractname}
    \end{center}
    \vspace{-0.4em}
}{%
    \end{minipage}
    \end{center}
    \vspace{0.8em}
}

\begin{document}
\maketitle

\begin{abstract}
We study subhomogeneity of the minimal operator system over a polyhedral
cone, or equivalently the size of irreducible Arveson boundary points of free
polyhedra. We obtain a complete classification in dimension three for cones
(equivalently, dimension two for polytopes): cones with three extreme rays
are \(1\)-subhomogeneous, cones with four extreme rays are
\(2\)-subhomogeneous, and cones with at least five extreme rays are not
subhomogeneous. In the last case, we construct irreducible Arveson boundary
points at every even matrix level. We prove that subhomogeneity passes to
faces and face quotients. We also determine the size of irreducible Arveson boundary points of free polyhedra over
products of two simplices, obtaining a dichotomy between the product of two
segments and all remaining cases. As applications, we classify
\(n\)-dimensional cones with \(n+1\) extreme rays. More generally, when every
facet omits at most two extreme rays, we show that subhomogeneity occurs
exactly for direct sums of simplicial cones and three-dimensional four-ray
cones. We also construct, for a range of fixed dimensions and ray counts,
cones exhibiting both subhomogeneous and non-subhomogeneous behavior.
\end{abstract}

\section{Introduction}

Finite-dimensional convex cones admit several natural extensions to matrix
levels. Operator systems provide an abstract framework for these extensions,
replacing scalar positivity by a compatible family of cones of
matrix-valued elements. For a given cone, the minimal operator system is one
canonical choice. A basic structural
question is whether this infinite hierarchy of matrix cones can nevertheless
be represented using matrices of some fixed size over a commutative
\(C^*\)-algebra. In the terminology used here, this asks whether the minimal
operator system is subhomogeneous. Such a bound is a strong finiteness
property: it says that positivity at all matrix levels is controlled by a
single finite level. Questions of finite-dimensional realization for
operator systems also arise naturally in spectrahedral containment and
semidefinite optimization, as well as in dilation and inclusion-constant
problems for matrix convex sets and free spectrahedra; see
\cite{davidsondoronshalitsolel,fritznetzerthom,heltonklepmcculloughschweighofer,
passershalitsolel}.

Recent work has developed several equivalent descriptions of
subhomogeneity. Subhomogeneous operator systems in connection
with systems generated by twisted commuting unitaries were defined and studied
in \cite{kiri}, while \cite{dannemuellernetzer} relates subhomogeneity to
the \(C^*\)-envelope, duality, and
noncommutative boundary conditions. On the matrix
convex side, the extreme-point and dilation theory of free spectrahedra shows
that Arveson boundary points play the role of classical extreme points:
compact free spectrahedra are generated by their irreducible Arveson boundary
points \cite{everthelton,evertheltonklepmccullough}. These results suggest a
concrete way to test subhomogeneity: determine whether the sizes of the
irreducible Arveson boundary points are uniformly bounded.

Polyhedral cones are a particularly natural testing ground for this
question. Their scalar geometry is finite and combinatorial, but their free extensions can exhibit genuinely noncommutative
behavior at matrix levels. The simplicial case is
commutative and hence \(1\)-subhomogeneous. At the next level of complexity,
recent results on noncommutative squares show that the quadrilateral case is
controlled by levels one and two
\cite{dannemuellernetzer,farenickmalekivarelasingla}; the Arveson boundary of
a free quadrilateral has also been studied explicitly in
\cite{evertprojective}. This makes it natural to ask how the number and
configuration of the extreme rays of a polyhedral cone govern
subhomogeneity.

Our first main result gives a complete answer in dimension three. A cone
with three extreme rays has a \(1\)-subhomogeneous minimal system, and a cone
with four extreme rays has a \(2\)-subhomogeneous minimal system. In sharp
contrast, every three-dimensional cone with at least five extreme rays is not subhomogeneous.
The proof is constructive: after a projective
normalization of the dual polygonal base, we build explicit irreducible Arveson boundary
points of unbounded size.

We then develop several mechanisms for transferring these conclusions to
higher dimensions. Subhomogeneity passes to faces, to quotients by faces,
and, with the same degree, to the factors of a direct sum. Consequently, a
three-dimensional face or face quotient with at least five extreme rays is
an obstruction to subhomogeneity. We also analyze free polyhedra over
products of two simplices. Their Arveson boundary points correspond to pairs of
projection-valued measures, so the problem is related to representations of
the free product \(\mathbb C^p*_{\mathbb C}\mathbb C^q\). This yields a
dichotomy: the product \(\Delta_1\times\Delta_1\) is generated at level two,
whereas every other product \(\Delta_{p-1}\times\Delta_{q-1}\), with
\(p,q\geqslant2\), has irreducible boundary points of unbounded size.

As an application, we classify all \(n\)-dimensional cones with exactly
\(n+1\) extreme rays: their minimal systems are subhomogeneous precisely for
the direct sums of a three-dimensional four-ray cone with a simplicial
factor, and in that case they are \(2\)-subhomogeneous. Using a combinatorial
result on polytopes, we obtain the same structural classification
for every cone whose facets omit at most two extreme rays: such a cone has a
subhomogeneous minimal system exactly when it is a direct sum of simplicial
cones and three-dimensional four-ray cones. Finally, we compare the two
behaviors for fixed dimension and fixed ray count. We construct
\(2\)-subhomogeneous examples with \(r\) rays whenever
\(n\leqslant r\leqslant n+\lfloor n/3\rfloor\), while
non-subhomogeneous examples exist for every \(r\geqslant n+1\) when
\(n\geqslant4\). Thus, over a substantial range, dimension and ray count
alone do not determine subhomogeneity.

\textbf{Disclaimer:} The AI system Prism by OpenAI was used to assist with
literature searches, typesetting, and the streamlining of proofs.
I take full responsibility for potential errors.

\section{Preliminaries}
Let \(C\subseteq V\) be a proper polyhedral cone in a finite-dimensional
real vector space, and fix an interior point \(u\in C\). The dual cone
\(C^\vee\) has the compact base
\[
    B=\{\varphi\in C^\vee\mid\varphi(u)=1\}
\]
of dimension \(\dim V-1\), which is a polytope.

For \(s\geqslant1\),
write \(\operatorname{Her}_s\) for the real vector space of
\(s\times s\) Hermitian matrices.
The \emph{minimal operator system over \(C\)} is the matrix-ordered
vector space
\[
    C^{\min}=\bigl(V,{(C_s^{\min})}_{s\geqslant1},u\bigr),
\]
where, under the natural identification
\(\operatorname{Her}_s(V)=\operatorname{Her}_s\otimes V\),
\[
    C_s^{\min}
    =\operatorname{cone}\bigl\{A\otimes c\mid
        A\in\operatorname{Her}_s,\ A\geqslant0,\ c\in C\bigr\}.
\]
Thus \(C_1^{\min}=C\), the matrix order unit at level \(s\) is
\(I_s\otimes u\), and \(C^{\min}\) is the smallest matrix ordering on
\(V\) whose scalar positive cone is \(C\); see, for example,
the discussion in~\cite[Section~3]{dannemuellernetzer}.

The \emph{(maximal) free polyhedron over
\(B\)} is the graded set \(B^{\rm ph}={(B^{\rm ph}(s))}_{s\geqslant1}\)
defined as follows. For
\(X=\sum_k A_k\otimes\varphi_k\in
\operatorname{Her}_s\otimes_{\mathbb R}V^*\) and \(c\in V\), set
\[
    X(c)=\sum_k\varphi_k(c)A_k\in\operatorname{Her}_s.
\]
Then
\[
\begin{aligned}
    B^{\rm ph}(s)
    &=\bigl\{X\in\operatorname{Her}_s\otimes_{\mathbb R}V^*\mid
       X(c)\geqslant0\text{ for every }c\in C,\ X(u)=I_s\bigr\}.
\end{aligned}
\]
At level one, the above definition gives
\(B^{\rm ph}(1)=B\). If
\(C=\operatorname{cone}(v_1,\ldots,v_r)\), then equivalently
\[
    B^{\rm ph}(s)
    =\bigl\{X\in\operatorname{Her}_s\otimes_{\mathbb R}V^*\mid
      X(u)=I_s,\ X(v_i)\geqslant0\text{ for }1\leqslant i\leqslant r\bigr\}.
\]
Thus \(B^{\rm ph}\) is a compact matrix convex set obtained by imposing
at every matrix level the facet inequalities of \(B\).

An operator system is called \emph{\(d\)-subhomogeneous} if it admits a
\(d\)-commutative realization, that is, a unital complete order embedding
into
\(\operatorname{Mat}_d(\mathcal A)\) for some commutative
\(C^*\)-algebra \(\mathcal A\). Equivalently, its \(C^*\)-envelope is a
\(d\)-subhomogeneous \(C^*\)-algebra, as shown
in~\cite[Theorem~4.1]{dannemuellernetzer}.

A point \(X\in B^{\rm ph}(s)\) is an \emph{Arveson boundary point} if,
whenever \(r\geqslant1\) and \(\widetilde X\in B^{\rm ph}(s+r)\) has the
block form
\[
    \widetilde X=
    \begin{pmatrix}X&\beta\\ \beta^*&\gamma\end{pmatrix},
\]
one has \(\beta=0\). It is \emph{irreducible} if the commutant of
\(\{X(c)\mid c\in V\}\) consists only of scalar matrices.

The ingredients of the following proposition are all known from existing work on
subhomogeneous operator systems and on Arveson and free extreme points;
see, in particular,~\cite{everthelton,evertheltonklepmccullough,kiri,kriel}.
We include the argument for completeness.

\begin{proposition}
\label{prop:bounded-level-generation}
    For every \(d\geqslant1\), the following are equivalent:
    \begin{enumerate}
        \item \(C^{\min}\) is \(d\)-subhomogeneous;
        \item \label{eq:bounded-level-generation}
        \(B^{\rm ph}=\operatorname{mconv}\bigl(B^{\rm ph}(d)\bigr)\);
        \item Every irreducible Arveson boundary point of \(B^{\rm ph}\) has size at most \(d\).
    \end{enumerate}
    Consequently, \(C^{\min}\) is subhomogeneous if and only if the sizes of
    the irreducible Arveson boundary points of \(B^{\rm ph}\) are uniformly
    bounded. Equivalently, the subhomogeneity degree is the smallest matrix
    size from which the whole free polyhedron can be recovered by direct sums
    and isometric compressions.
\end{proposition}

\begin{proof}
    By the minimal--maximal duality criterion
    in~\cite[Theorem~4.1]{dannemuellernetzer}, condition (1) holds exactly when
    \({(C^\vee)}^{\max}\) is generated, as a matrix convex cone, by its levels
    of size at most \(d\). We show that this is equivalent to (2).

    Suppose first that the dual cone is generated by its first \(d\) levels,
    and take \(X\in B^{\rm ph}(s)\). Write
    \[
        X=\sum_j V_j^*Y_j V_j,
        \qquad Y_j\in{(C^\vee)}^{\max}_{t_j},\quad t_j\leqslant d.
    \]
    Put \(P_j=Y_j(u)\geqslant0\). Since \(u\) is an interior point of \(C\),
    for every \(c\in V\) there is \(\lambda>0\) such that
    \(\lambda u\pm c\in C\). Hence
    \(\lambda P_j\pm Y_j(c)\geqslant0\), which implies that every \(Y_j(c)\)
    vanishes on \(\ker P_j\). After splitting off this common zero block,
    replace \(Y_j\) by
    \[
        \widehat Y_j=P_j^{-1/2}Y_j P_j^{-1/2}
    \]
    and \(V_j\) by \(P_j^{1/2}V_j\). We may therefore assume that
    \(Y_j(u)=I_{t_j}\). Evaluating at \(u\) gives
    \(\sum_j V_j^*V_j=I_s\), so \(X\) belongs to the matrix convex hull of the
    first \(d\) levels. A point of a smaller level can be padded to level
    \(d\) by taking its direct sum with scalar points of \(B\), and then
    recovered by compression. Thus (2) follows.

    Conversely, assume (2), let \(Y\in{(C^\vee)}^{\max}_s\), and put
    \(P=Y(u)\). The same argument shows that \(Y\) vanishes on \(\ker P\).
    On \(\operatorname{ran}P\), the normalized point
    \(\widehat Y=P^{-1/2}YP^{-1/2}\) belongs to \(B^{\rm ph}\). A matrix convex decomposition of \(\widehat Y\) by points of size \(d\), followed
    by multiplication by \(P^{1/2}\), gives a matrix-conic decomposition of
    \(Y\) by elements of the \(d\)-th level of \({(C^\vee)}^{\max}\). This proves (1).

    Now suppose (2) holds and let \(X\in B^{\rm ph}(s)\) be an irreducible
    Arveson boundary point. Choose a matrix convex decomposition
    \[
        X=\sum_j V_j^*Y^{(j)}V_j,
        \qquad Y^{(j)}\in B^{\rm ph}(d),
        \qquad \sum_j V_j^*V_j=I_s.
    \]
    Let \(\Gamma:\mathbb C^s\to\bigoplus_j\mathbb C^{d}\) be the isometry with
    components \(V_j\), and put \(Y=\bigoplus_j Y^{(j)}\). Then \(X\) is the
    compression of \(Y\) to \(\Gamma\mathbb C^s\). After a unitary
    conjugation, the Arveson boundary property implies that this subspace
    reduces every coordinate of \(Y\), and hence
    \[
        Y^{(j)}(c)V_j=V_j X(c)\qquad(c\in V).
    \]
    It follows that \(V_j^*V_j\) commutes with every coordinate of \(X\).
    By irreducibility,
    \[
        V_j^*V_j=\lambda_j I_s
    \]
    for some \(\lambda_j\geqslant0\). Whenever \(\lambda_j>0\), the map
    \(V_j\) has rank \(s\), so \(d\geqslant s\). Since at least one
    \(\lambda_j\) is nonzero, we obtain (3).

    Conversely, suppose (3) holds. Since \(B\) is a polytope,
    \(B^{\rm ph}\) is a bounded free spectrahedron defined by a pencil with
    real scalar coefficients; in particular, it is closed under complex
    conjugation. It is therefore the matrix convex hull of its irreducible
    Arveson boundary points by~\cite[Theorems~1.1 and~1.2]{everthelton}.
    By assumption, all these points
    have size at most \(d\), and the padding argument above shows that
    \(B^{\rm ph}\) is generated by its \(d\)-th level. Thus (2) holds.
\end{proof}

\section{Cones in dimension three}

We begin with the first nontrivial dimension. A three-dimensional
proper cone is determined, up to the choice of an affine section, by a
convex polygon. A cone with three extreme rays is simplicial, so its minimal
operator system is \(1\)-subhomogeneous. Every cone with four extreme rays is
isomorphic to the cone over a quadrilateral and has a \(2\)-subhomogeneous
minimal system; see~\cite[Corollary~3.3]{farenickmalekivarelasingla} and
~\cite[Theorem~6.1(ii)]{dannemuellernetzer}, as well as
\Cref{cor:n-plus-one} below. The following result shows that the passage from
four to five extreme rays is decisive: from five rays onward, the associated
free polyhedron has irreducible Arveson boundary points at arbitrarily large
levels.

\begin{theorem}
\label{thm:five-ray}
    Let \(C\) be a proper three-dimensional polyhedral cone with at least five
    extreme rays; equivalently, let \(B\) be a convex polygon with at least
    five vertices. Then, for every \(m\geqslant1\), the level
    \(B^{\rm ph}(2m)\) contains an irreducible Arveson boundary point.
    Consequently, \(B^{\rm ph}\) is not generated by any fixed finite level,
    and \(C^{\min}\) is not subhomogeneous.
\end{theorem}

\begin{proof}
    After an invertible linear change in
    the homogenized coordinates and a corresponding choice of affine slice,
    we may assume that \(B\) is contained in the pentagon
    \[
        \widetilde B=\widetilde B_{\alpha,\beta}
        =\bigl\{(x,y)\mid 0\leqslant x,y\leqslant1,\
        \alpha x+\beta y\geqslant1\bigr\},
        \qquad \alpha,\beta>1,
    \]
    and is obtained from \(\widetilde B\) by further inequalities which
    only cut the corner \((0,1/\beta)\). More precisely, every additional
    inequality is strict at
    \[
        p_1=(\tfrac1\alpha,0),
        \qquad p_2=(1,0),
        \qquad p_3=(0,1).
    \]
    Hence, for
    \[
        T=\operatorname{conv}\{p_1,p_2,p_3\},
    \]
    some relative neighborhood of \(T\) in \(\widetilde B\) is contained
    in \(B\). The change of homogenized coordinates induces a projective
    transformation of the base, which preserves the Arveson boundary and
    irreducibility~\cite[Theorem~3.6 and Lemma~3.2]{evertprojective}.

    The maximal free polyhedron over
    \(\widetilde B\) is
    \[
        \widetilde B^{\rm ph}(s)=
        \bigl\{(X,Y)\in{\operatorname{Her}_s(\mathbb C)}^2\mid
        0\leqslant X,Y\leqslant I,
        \ \alpha X+\beta Y-I\geqslant0\bigr\}.             \tag{1}
    \]

    Fix \(m\geqslant1\), put \(H=\bigoplus_{j=1}^m\mathbb C^2\), and choose
    distinct numbers
    \[
        0<c_1,\ldots,c_m<
        \frac{(\alpha-1)(\beta-1)}{\alpha\beta}.
    \]
    On the \(j\)-th summand of \(H\), let
    \[
        P_j=\begin{pmatrix}1&0\\0&0\end{pmatrix},
        \qquad
        Q_j=
        \begin{pmatrix}
            c_j&\sqrt{c_j(1-c_j)}\\
            \sqrt{c_j(1-c_j)}&1-c_j
        \end{pmatrix},
    \]
    and set
    \[
        P=\bigoplus_{j=1}^m P_j,
        \qquad Q=\bigoplus_{j=1}^m Q_j,
        \qquad M=\alpha P+\beta Q-I.
    \]
    Both \(P_j\) and \(Q_j\) are rank-one projections. On the \(j\)-th block,
    \[
        \det(M_j)=(\alpha-1)(\beta-1)-\alpha\beta c_j>0,
        \qquad \operatorname{tr}(M_j)=\alpha+\beta-2>0.
    \]
    Hence \(M>0\).

    Let \(e_j=(1,0)\) in the \(j\)-th block, and set
    \[
        r=\frac1{\sqrt m}\sum_{j=1}^m e_j,
        \qquad R=rr^*.
    \]
    Define
    \[
        a=r^*M^{-1}r,
        \qquad t=\frac1{\alpha a},
        \qquad X=P-tR,
        \qquad Y=Q.
    \]
    The upper-left entry of \(M_j^{-1}\) is
    \[
        \frac{\beta(1-c_j)-1}
        {(\alpha-1)(\beta-1)-\alpha\beta c_j}
        >\frac1\alpha.
    \]
    Since \(r\) has a component in each first coordinate, this gives
    \(a>1/\alpha\), and therefore \(0<t<1\). Thus
    \(0\leqslant X\leqslant I\), while \(Y\) is a projection. Moreover,
    the rank-one perturbation criterion yields
    \[
        \alpha X+\beta Y-I=M-\frac1aR\geqslant0,
    \]
    and equality in that criterion gives
    \[
        \ker(\alpha X+\beta Y-I)=\mathbb C f,
        \qquad f=M^{-1}r.
    \]
    Set \(X_m=X\) and \(Y_m=Y\). Hence
    \(Z_m=(X_m,Y_m)\) belongs to \(\widetilde B^{\rm ph}(2m)\).

    We now choose the parameters sufficiently small to ensure that
    \(Z_m\in B^{\rm ph}(2m)\). Put \(\delta=\max_j\{c_j\}\). Since
    \[
        a=\frac1m\sum_{j=1}^m
        \frac{\beta(1-c_j)-1}
        {(\alpha-1)(\beta-1)-\alpha\beta c_j},
    \]
    for \(\delta\to0\) one has
    \[
        Y\to Y^{(0)}=I-P,
        \qquad
        a\to\frac1{\alpha-1},
        \qquad
        t\to\frac{\alpha-1}{\alpha},
    \]
    and hence
    \[
        X\to
        X^{(0)}=P-\frac{\alpha-1}{\alpha}R.
    \]
    The limiting pair \((X^{(0)},Y^{(0)})\) is commuting and has joint
    spectrum contained in
    \[
        \bigl\{(\tfrac1\alpha,0),(1,0),(0,1)\bigr\}.
    \]
    Its joint numerical range is therefore contained in \(T\).
    Thus, by choosing all the distinct \(c_j\)'s
    sufficiently small,
    the joint numerical range of \((X,Y)\) lies in the relative neighborhood
    of \(T\) contained in \(B\). Equivalently, every facet inequality
    of \(B\) holds at the matrix level, and hence
    \[
        Z_m\in B^{\rm ph}(2m).
    \]

    We now show that \(Z_m\) is an Arveson boundary point. Suppose
    \[
        \widetilde X=\begin{pmatrix}X&\xi\\\xi^*&X_0\end{pmatrix},
        \qquad
        \widetilde Y=\begin{pmatrix}Y&\eta\\\eta^*&Y_0\end{pmatrix}
    \]
    is a dilation belonging to \(B^{\rm ph}\). Since the five inequalities
    defining \(\widetilde B\) are valid on \(B\), this dilation satisfies
    \((1)\). For a positive block matrix, the range of an off-diagonal block
    is contained in the range of the corresponding diagonal block. Applying
    this to \(\widetilde X\) and \(I-\widetilde X\), and using \(0<t<1\), gives
    \[
        \operatorname{ran}(\xi)
        \subseteq\operatorname{ran}(X)\cap\operatorname{ran}(I-X)
        =\mathbb Cr.
    \]
    Since \(Y\) is a projection, the same argument applied to
    \(\widetilde Y\) and \(I-\widetilde Y\) gives \(\eta=0\). Finally,
    positivity of \(\alpha\widetilde X+\beta\widetilde Y-I\) implies
    \[
        \operatorname{ran}(\alpha\xi+\beta\eta)
        \subseteq\operatorname{ran}(\alpha X+\beta Y-I)=f^\perp.
    \]
    But \(\operatorname{ran}(\xi)\subseteq\mathbb Cr\), \(\eta=0\), and
    \[
        \langle f,r\rangle=r^*M^{-1}r=a>0.
    \]
    Thus \(\xi=0\), proving the Arveson boundary property.

    It remains to prove irreducibility. A direct calculation gives
    \[
        R=\frac{X(I-X)}{t(1-t)}.
    \]
    Hence an operator commuting with \(X\) also commutes with \(R\) and with
    \(P=X+tR\). The pairs \((P_j,Q_j)\) are irreducible, and the distinct
    parameters \(c_j\) make them pairwise inequivalent. Therefore
    \[
        \{P,Q\}'=
        \left\{\bigoplus_{j=1}^m\lambda_j I_2\mid
        \lambda_j\in\mathbb C\right\}.
    \]
    Commutation with \(R=rr^*\) forces all \(\lambda_j\) to be equal, because
    the component of \(r\) in every block is nonzero. Consequently,
    \[
        \{X,Y\}'=\mathbb C I_{2m},
    \]
    so \(Z_m\) is irreducible.
\end{proof}

\section{Higher dimensions}

We next transfer the three-dimensional obstruction to cones of arbitrary
dimension. Two natural lower-dimensional objects associated with a cone are
its faces and its face quotients. The next two propositions show that
subhomogeneity passes to both.

\begin{proposition}
\label{prop:faces}
    Let \(C\) be a proper polyhedral cone and \(F\) a face of \(C\). If
    \(C^{\min}\) is \(d\)-subhomogeneous, then \(F^{\min}\) is \(d\)-subhomogeneous.
\end{proposition}

\begin{proof}
    Put \(W=\operatorname{span}F\). Since \(C\) is polyhedral, the face
    \(F\) is exposed. The restriction argument
    from~\cite[Lemma~4.6]{fritznetzerthom} therefore gives
    \[
        C_s^{\min}\cap\operatorname{Her}_s(W)=F_s^{\min}
        \qquad(s\geqslant1).                              \tag{2}
    \]

    Suppose that \(C^{\min}\) is \(d\)-subhomogeneous.
    By~\cite[Theorem~4.1]{dannemuellernetzer}, it is \(d\)-maximal. Let
    \(X\in\operatorname{Her}_s(W)\), and assume that
    \[
        M^*XM\in F_d^{\min}
        \qquad\text{for every }M\in\operatorname{Mat}_{s,d}(\mathbb C).
    \]
    By \((2)\), all these compressions belong to \(C_d^{\min}\). The
    \(d\)-maximality of \(C^{\min}\) therefore implies that
    \(X\in C_s^{\min}\). Since \(X\in\operatorname{Her}_s(W)\), another
    application of \((2)\) gives \(X\in F_s^{\min}\). Thus \(F^{\min}\) is
    \(d\)-maximal and hence, again
    by~\cite[Theorem~4.1]{dannemuellernetzer}, \(d\)-subhomogeneous.
\end{proof}

Passing to a face restricts the cone to a subspace. The complementary
operation is to collapse the span of a face. For an extreme ray, the latter
operation produces the homogenized form of the usual vertex figure. This
gives a second way to locate lower-dimensional obstructions inside a
higher-dimensional cone.

\begin{proposition}
\label{prop:face-figures}
    Let \(F\) be a face of a proper polyhedral cone \(C\subseteq V\), let
    \(q:V\to V/\operatorname{span}F\) be the quotient map, and put
    \[
        C/F=q(C).
    \]
    Then \(C/F\) is a proper polyhedral cone. If \(C^{\min}\) is
    \(d\)-subhomogeneous, then \({(C/F)}^{\min}\) is
    \(d\)-subhomogeneous.

    In particular, if \(R\) is an extreme ray of \(C\), then the conclusion
    applies to the vertex-figure cone \(C/R\).
\end{proposition}

\begin{proof}
    Write \(W=\operatorname{span}F\). Since \(F\) is a face, the quotient
    cone \(q(C)\) is pointed. Indeed, if \(q(c_1)=-q(c_2)\) with
    \(c_1,c_2\in C\), then \(c_1+c_2\in F\), so the facial property gives
    \(c_1,c_2\in F\). It is polyhedral and spans \(V/W\), hence is proper.

    The dual cone \({(C/F)}^\vee\) identifies with
    \(C^\vee\cap W^\perp\). At every matrix level this identifies
    \(\bigl({(C/F)}^\vee\bigr)^{\max}\) with the matrix face
    \[
        \mathcal G_s=
        \bigl\{X\in{(C^\vee)}^{\max}_s\mid
        X(f)=0\text{ for every }f\in F\bigr\}.          \tag{3}
    \]

    Suppose that \(C^{\min}\) is \(d\)-subhomogeneous. By
    \Cref{prop:bounded-level-generation} and minimal--maximal duality,
    \({(C^\vee)}^{\max}\) is generated by its levels of size at most \(d\).
    The same is true for the matrix face \(\mathcal G\). Indeed, decompose a
    point of \(\mathcal G_s\) using summands of size at most \(d\). Evaluating
    the decomposition at any \(f\in F\) gives a sum of positive semidefinite
    matrices equal to zero. Consequently, after compressing each summand to
    the range of its coefficient map, every summand belongs to \(\mathcal G\)
    and the decomposition is unchanged. Thus
    \(\bigl({(C/F)}^\vee\bigr)^{\max}\) is generated by its levels of size at
    most \(d\), and minimal--maximal duality shows that
    \({(C/F)}^{\min}\) is \(d\)-subhomogeneous.

    For an extreme ray \(R\), this construction corresponds to taking the
    usual vertex figure, up to a projective transformation.
\end{proof}

\begin{corollary}
\label{cor:face-obstructions}
    Let \(C\) be a proper polyhedral cone of arbitrary finite dimension. If
    \(C\) has a three-dimensional face or face quotient with at least five
    extreme rays, then \(C^{\min}\) is not subhomogeneous.
\end{corollary}

\begin{proof}
    This is the contrapositive of
    \Cref{prop:faces,prop:face-figures}, combined with \Cref{thm:five-ray}.
\end{proof}

Faces and quotients propagate negative results. For the constructions below
we also need a positive result: direct sums preserve the exact
subhomogeneity bound.

\begin{proposition}
\label{prop:permanence}
    Let \(C\) and \(D\) be proper polyhedral cones. Then
    \({(C\oplus D)}^{\min}\) is \(d\)-subhomogeneous if and only if both
    \(C^{\min}\) and \(D^{\min}\) are \(d\)-subhomogeneous.
\end{proposition}

\begin{proof}
    At every matrix level,
    \[
        {(C\oplus D)}^{\min}_s=C_s^{\min}\oplus D_s^{\min}.       \tag{4}
    \]
    We again use that \(d\)-subhomogeneity is equivalent to
    \(d\)-maximality.

    So suppose first that \(C^{\min}\) and \(D^{\min}\) are \(d\)-maximal.
    Let \(X=(X_C,X_D)\) be of size \(s\), and assume that every compression
    of \(X\) to size \(d\) belongs to \({(C\oplus D)}^{\min}_d\). By
    \((4)\), every such compression of \(X_C\) belongs to \(C_d^{\min}\),
    and similarly for \(X_D\). The \(d\)-maximality of the two factors gives
    \(X_C\in C_s^{\min}\) and \(X_D\in D_s^{\min}\), so \((4)\) yields
    \(X\in{(C\oplus D)}^{\min}_s\). Thus the direct sum is \(d\)-maximal.

    Conversely, suppose that \({(C\oplus D)}^{\min}\) is \(d\)-maximal.
    If every size-\(d\) compression of \(X_C\) belongs to \(C_d^{\min}\),
    then every such compression of \((X_C,0)\) belongs to
    \({(C\oplus D)}^{\min}_d\). Hence \((X_C,0)\) belongs to
    \({(C\oplus D)}^{\min}_s\), and \((4)\) implies
    \(X_C\in C_s^{\min}\). Thus \(C^{\min}\) is \(d\)-maximal, and the
    same argument applies to \(D^{\min}\).
\end{proof}

The next family supplies basic nonsimplicial building blocks. Matrix
points over a product of simplices are pairs of positive operator-valued
measures. Their Arveson boundary points therefore consist of pairs of
projection-valued measures, equivalently, representations of a unital free
product of finite-dimensional commutative \(C^*\)-algebras. This connects the
matrix-convex boundary theory with the representation theory studied
in~\cite{dykematorres,exelloring}.

\begin{theorem}
\label{thm:product-simplices}
    For \(p,q\geqslant2\), let
    \[
        B_{p,q}=\Delta_{p-1}\times\Delta_{q-1}
    \]
    and let \(B_{p,q}^{\rm ph}\) be the maximal free polyhedron over
    \(B_{p,q}\). Explicitly,
    \[
    \begin{aligned}
        B_{p,q}^{\rm ph}(s)=\bigl\{(E_1,\ldots,E_p,F_1,\ldots,F_q)\mid\;&
        E_i,F_j\in\operatorname{Her}_s(\mathbb C),\quad E_i,F_j\geqslant0,\\
        &\sum_{i=1}^p E_i=I_s,\quad \sum_{j=1}^q F_j=I_s\bigr\}.
    \end{aligned}
    \]
    Its Arveson boundary points are precisely the pairs of
    projection-valued measures, that is, the points for which every \(E_i\)
    and every \(F_j\) is a projection.

    The following dichotomy holds:
    \begin{enumerate}
        \item If \(p=q=2\), then
        \[
            B_{2,2}^{\rm ph}
            =\operatorname{mconv}\bigl(B_{2,2}^{\rm ph}(2)\bigr),
        \]
        and every irreducible Arveson boundary point of
        \(B_{2,2}^{\rm ph}\) has size at most \(2\).
        \item If \(\max\{p,q\}\geqslant3\), then
        \(B_{p,q}^{\rm ph}\) has irreducible Arveson boundary points at
        unbounded matrix levels and is not generated by any fixed finite level.
    \end{enumerate}
\end{theorem}

\begin{proof}
    For a single simplex, the assertion that the Arveson boundary points are
    precisely the PVMs is the free-simplex boundary description
    in~\cite[Theorem~6.5]{evertheltonklepmccullough}. This implies the same
    statement for the product.

    Pairs of PVMs with \(p\) and \(q\) outcomes are the representations of
    \[
        A_{p,q}=\mathbb C^p*_{\mathbb C}\mathbb C^q,
    \]
    and the notions of irreducibility agree. The algebra \(A_{p,q}\) is
    residually finite-dimensional by~\cite{exelloring}.

    If \(p=q=2\), then \(A_{2,2}\) is generated by two projections, so every
    irreducible representation has dimension at most \(2\) by Halmos'
    two-projections theorem~\cite{halmos}. This proves the first alternative.

    If \(\max\{p,q\}\geqslant3\), the algebra \(A_{p,q}\) is primitive
    by~\cite{dykematorres}. Thus it has a faithful irreducible
    representation, which must be infinite dimensional, since \(A_{p,q}\) is
    infinite-dimensional. Hence \(A_{p,q}\) is not
    subhomogeneous. If its finite-dimensional irreducible representations
    had uniformly bounded size, residual finite-dimensionality would embed
    \(A_{p,q}\) into a product of matrix algebras of uniformly bounded size,
    forcing \(A_{p,q}\) itself to be subhomogeneous. Hence it has
    finite-dimensional irreducible representations of unbounded size.
\end{proof}

We now apply this to cones with exactly one more extreme ray than
their dimension. Such a cone has a unique linear dependence among its ray
generators. The positive and negative parts of that dependence determine a
product of two simplices, while the generators with zero coefficient split
off simplicial factors.

Let \(C_{\square}\) denote a three-dimensional polyhedral convex cone with
four extreme rays. Since they are all isomorphic, the explicit choice does
not matter.

\begin{corollary}
\label{cor:n-plus-one}
    Let \(n\geqslant3\), and let \(C\) be an \(n\)-dimensional proper
    polyhedral cone with \(n+1\) extreme rays. Then \(C^{\min}\) is
    subhomogeneous if and only if
    \[
        C\cong C_{\square}\oplus\mathbb R_+^{n-3}.
    \]
    In this case \(C^{\min}\) is \(2\)-subhomogeneous.

    Consequently, every three-dimensional proper polyhedral cone with four extreme rays
    has a \(2\)-subhomogeneous minimal system, whereas for every
    \(n\geqslant4\) there is an \(n\)-dimensional proper polyhedral cone with
    \(n+1\) extreme rays whose minimal system is not subhomogeneous.
\end{corollary}

\begin{proof}
    Choose extreme-ray generators \(v_1,\ldots,v_{n+1}\). Up to a nonzero
    scalar, they satisfy a unique relation
    \[
        \sum_{i=1}^{n+1}\alpha_i v_i=0.
    \]
    Write
    \[
        P=\{i\mid\alpha_i>0\},\qquad
        N=\{i\mid\alpha_i<0\},\qquad
        Z=\{i\mid\alpha_i=0\},
    \]
    and put \(p=|P|\), \(q=|N|\), and \(z=|Z|\). Since \(C\) is pointed,
    the relation has coefficients of both signs. Neither \(P\) nor \(N\)
    can be a singleton, since then the corresponding ray generator would be
    a positive linear combination of the generators on the other side,
    contradicting extremality. Thus \(p,q\geqslant2\).

    Let \(C_0\) be the cone generated
    by the rays indexed by \(P\cup N\). The rays indexed by \(Z\) are linearly
    independent modulo \(\operatorname{span}(C_0)\), since otherwise there
    would be a second relation among all the ray generators. Hence
    \[
        C\cong C_0\oplus\mathbb R_+^z.                       \tag{5}
    \]

    After rescaling the generators in the unique relation, the maximal free
    polyhedron associated with \(C_0\) is the one over
    \(\Delta_{p-1}\times\Delta_{q-1}\): its points are precisely pairs of
    POVMs \((E_1,\ldots,E_p)\) and \((F_1,\ldots,F_q)\). By
    \Cref{thm:product-simplices}, \Cref{prop:bounded-level-generation} and \Cref{prop:permanence},
    the minimal system over \(C\) is subhomogeneous exactly when \(p=q=2\).
    In that case \(z=n-3\) and \(C_0\cong C_{\square}\), which gives the
    displayed decomposition and shows that the system is
    \(2\)-subhomogeneous.

    For \(n=3\), necessarily \(p=q=2\) and \(z=0\), proving the first
    consequence. For \(n\geqslant4\), take the four-dimensional five-ray cone
    associated with \(\Delta_1\times\Delta_2\) and form its direct sum with
    \(\mathbb R_+^{n-4}\). This cone has dimension \(n\), exactly
    \(5+(n-4)=n+1\) extreme rays, and a non-subhomogeneous minimal system by
    \Cref{thm:product-simplices} and \Cref{prop:permanence}.
\end{proof}

Finally, the preceding ingredients can be combined to compare the two
possible behaviors at fixed dimension and depending on the number of rays.

\begin{theorem}
\label{thm:ray-counts}
Let \(n\geqslant1\).
\begin{enumerate}
    \item For every integer
    \[
        n\leqslant r\leqslant n+\lfloor n/3\rfloor,
    \]
    there is an \(n\)-dimensional proper polyhedral cone with exactly
    \(r\) extreme rays whose minimal operator system is
    \(2\)-subhomogeneous. When \(r=n\), it is in fact
    \(1\)-subhomogeneous.
    \item Non-subhomogeneous examples with exactly \(r\) extreme rays
    exist for
    \[
    \begin{cases}
        n=3,\quad r\geqslant5,\\
        n\geqslant4,\quad r\geqslant n+1.
    \end{cases}
    \]
    \end{enumerate}
    Consequently, whenever
    \[
        n\geqslant4,
        \qquad
        n+1\leqslant r\leqslant n+\lfloor n/3\rfloor,
    \]
    there are \(n\)-dimensional polyhedral cones with exactly \(r\) extreme rays
    exhibiting both behaviors.
\end{theorem}

\begin{proof}
    For (1), put \(k=r-n\) and set
    \[
        S_{n,r}=C_{\square}^{\oplus k}
        \oplus\mathbb R_+^{n-3k}.
    \]
    The assumption \(k\leqslant\lfloor n/3\rfloor\) makes this well defined, and
    \[
        \dim S_{n,r}=3k+(n-3k)=n,
        \qquad
        \#\operatorname{ExtRays}(S_{n,r})
        =4k+(n-3k)=n+k=r.
    \]
    If \(k=0\), the cone is simplicial. If \(k\geqslant1\), its minimal
    system is \(2\)-subhomogeneous by \Cref{prop:permanence}.

    For (2), the case \(n=3\) follows from \Cref{thm:five-ray}. Now let
    \(n\geqslant4\). When \(r=n+1\), the result follows from \Cref{cor:n-plus-one}.
    Finally, suppose \(r\geqslant n+2\) and put
    \[
        q=r-n+3\geqslant5.
    \]
    Choose any three-dimensional proper cone \(C^{(q)}\) with \(q\) extreme
    rays and set
    \[
        C_{n,r}=C^{(q)}\oplus\mathbb R_+^{n-3}.
    \]
    Then
    \[
        \dim C_{n,r}=n,
        \qquad
        \#\operatorname{ExtRays}(C_{n,r})=q+n-3=r.
    \]
    By \Cref{thm:five-ray}, \({(C^{(q)})}^{\min}\) is not subhomogeneous, and
    \Cref{prop:permanence} passes this obstruction to \(C_{n,r}^{\min}\).
\end{proof}

The preceding constructions provide
\(n+\lfloor n/3\rfloor\) rays by taking direct sums of square cones and
simplicial factors. By a result from~\cite{wernerwotzlaw}, the same
building blocks are forced under a specific incidence hypothesis.

\begin{theorem}
\label{thm:facet-complements}
    Let \(C\) be an \(n\)-dimensional proper polyhedral cone, and suppose
    that every facet of \(C\) omits at most two extreme rays. Then the
    following are equivalent:
    \begin{enumerate}
        \item \(C^{\min}\) is subhomogeneous;
        \item there are integers \(a,m\geqslant0\) such that
        \[
            C\cong\mathbb R_+^a\oplus C_{\square}^{\oplus m}.
        \]
    \end{enumerate}
    In this case \(C^{\min}\) is \(2\)-subhomogeneous and
    \[
        \#\operatorname{ExtRays}(C)
        =n+m\leqslant n+\lfloor n/3\rfloor.
    \]
\end{theorem}

\begin{proof}
    Let \(P\) be a compact base of \(C\). The hypothesis says that every
    facet of \(P\) omits at most two vertices.
    By~\cite[Lemma~3.8]{wernerwotzlaw}, \(P\) is an iterated pyramid over a
    join of direct sums of two simplices: more precisely, it has the form
    \[
        \Delta_{a-1}*
        (\Delta_{j_1}\oplus\Delta_{k_1})*\cdots *
        (\Delta_{j_m}\oplus\Delta_{k_m}),
        \qquad j_i,k_i\geqslant1.
    \]
    The factors occur as faces of \(P\). Taking their homogenized cones
    turns the join into a direct sum, and therefore
    \[
        C\cong\mathbb R_+^a\oplus D_1\oplus\cdots\oplus D_m,
    \]
    where \(D_i\) has dimension \(j_i+k_i+1\), has one more extreme ray
    than its dimension, and its unique linear dependence has no zero
    coefficient.

    By \Cref{prop:permanence}, \(C^{\min}\) is subhomogeneous if and only
    if every \(D_i^{\min}\) is subhomogeneous. Applying
    \Cref{cor:n-plus-one} to \(D_i\), and using the fact that its unique
    dependence has full support, shows that this happens precisely when
    \(j_i=k_i=1\). In that case \(D_i\cong C_{\square}\). This proves the
    equivalence and the \(2\)-subhomogeneity assertion.

    Finally, the displayed decomposition gives
    \[
        n=a+3m,
        \qquad
        \#\operatorname{ExtRays}(C)=a+4m=n+m,
    \]
    and hence \(m\leqslant\lfloor n/3\rfloor\).
\end{proof}

\section{Open problems}

The preceding results suggest that the examples constructed from simplicial
cones and three-dimensional four-ray cones may in fact exhaust the
subhomogeneous case. We therefore propose the following claim:

\begin{equation}
\label{eq:classification-conjecture}
    C^{\min}\text{ is subhomogeneous}
    \quad\Longleftrightarrow\quad
    C\cong\mathbb R_+^a\oplus C_{\square}^{\oplus m}
    \text{ for some }a,m\geqslant0.                         \tag{6}
\end{equation}
The result is proved in \Cref{thm:facet-complements} under the
additional assumption that every facet omits at most two extreme rays.

A proof of~\eqref{eq:classification-conjecture} would also settle two more conjectures motivated the results above. First, every subhomogeneous minimal
operator system over a polyhedral cone would already be
\(2\)-subhomogeneous. Second, the optimal universal threshold for the number
of extreme rays would be
\[
    r(n)=n+\lfloor n/3\rfloor+1
    \qquad(n\geqslant3).
\]
Indeed, \Cref{thm:ray-counts} provides subhomogeneous cones with
\(n+\lfloor n/3\rfloor\) extreme rays,
whereas~\eqref{eq:classification-conjecture} would imply that no subhomogeneous
\(n\)-dimensional cone can have more.

\end{document}